\documentclass[11pt]{elsarticle}     

\usepackage{hyperref}

\usepackage{tikz}
\usepackage{amstext,amsfonts,amsmath,amssymb,array}
\usepackage{amsthm,moresize,mathtools,relsize}
\usepackage{multirow}
\usepackage{fullpage}
\makeatletter
\newcommand*{\transpose}{%
  {\mathpalette\@transpose{}}%
}
\newcommand*{\@transpose}[2]{%
  \raisebox{\depth}{$\m@th#1\intercal$}%
}
\def\@author#1{\g@addto@macro\elsauthors{\normalsize%
\def\baselinestretch{1}%
\upshape\authorsep#1\unskip\textsuperscript{%
\ifx\@fnmark\@empty\else\unskip\sep\@fnmark\let\sep=,\fi
\ifx\@corref\@empty\else\unskip\sep\@corref\let\sep=,\fi}%
\def\authorsep{\unskip,\space}%
\global\let\@fnmark\@empty
\global\let\@corref\@empty  
\global\let\sep\@empty}%
\@eadauthor={#1}
}
\makeatother

\def\dj{d\kern-0.4em\char"16\kern-0.1em}

\DeclareMathOperator{\tr}{tr}
\DeclareMathOperator{\bd}{bd}

\DeclareMathOperator{\md}{mod}

\newtheorem{theorem}{Theorem}[section]
\newtheorem{lemma}[theorem]{Lemma}
\newtheorem{corollary}[theorem]{Corollary}

\theoremstyle{definition}
\newtheorem{example}[theorem]{Example}

\newtheorem{question}[theorem]{Question}

\usepackage{pgf}
\usepackage{tikz}
\usetikzlibrary{decorations.pathreplacing}
\usetikzlibrary{arrows,decorations.markings,shapes}
\tikzstyle{vertex}=[circle, draw, fill=black, inner sep=0pt, minimum width=4pt]
\tikzstyle{pedge}=[draw,-]
\tikzstyle{nedge}=[draw,densely dashed]

\journal{the journal.}

\begin{document}

\begin{frontmatter}


\title{\bf Signed $(0,2)$-graphs with few eigenvalues and \\ a symmetric spectrum }

\author{Gary R. W. Greaves}
\ead{gary@ntu.edu.sg}
\address{School of Physical and Mathematical Sciences,
  Nanyang Technological University \\
   21 Nanyang Link, Singapore 637371}

\author{Zoran Stani\' c}
\ead{zstanic@math.rs}
\address{Faculty of Mathematics, University of Belgrade \\
Studentski trg 16, 11 000 Belgrade, Serbia}
\begin{abstract} 
We investigate properties of signed graphs that have few distinct eigenvalues together with a symmetric spectrum.
Our main contribution is to determine all signed $(0,2)$-graphs with vertex degree at most $6$ that have precisely two distinct eigenvalues $\pm \lambda$.
Next, we consider to what extent induced subgraphs of signed graph with two distinct eigenvalues $\pm \lambda$ are determined by their spectra.
Lastly, we classify signed $(0,2)$-graphs that have a symmetric spectrum with three distinct eigenvalues and give a partial classification for those with four distinct eigenvalues.
\end{abstract}

\begin{keyword}
  {signed $(0, 2)$-graph \sep rectagraph \sep adjacency matrix \sep symmetric spectrum \sep weighing matrix \sep bipartite double\sep (folded) cube} 

\MSC 05C22 \sep 05C50
\end{keyword}

\end{frontmatter}

\section{Introduction}

A \textit{signed graph} $\dot{G}=(G, \sigma)$ is comprised of a graph $G=(V, E)$ (called the \emph{underlying graph} {of $\dot G$}) together with a \textit{signature} function $\sigma\colon E\longrightarrow\{1, -1\}$. The number of vertices of $\dot{G}$ is called the \textit{order} and  {is} denoted by $n$. 
 {An edge $e \in E$ of $\dot G$ is called \emph{positive} (resp. \emph{negative}) if $\sigma(e) = +1$ (resp. $\sigma(e) = -1$).}
We  { can} interpret an  {(unsigned)} graph as a signed graph with all {its} edges being positive.

The \textit{adjacency matrix} $A_{\dot{G}}$ of $\dot{G}$ is obtained from the adjacency matrix of its underlying graph by reversing the sign of all 1s
that correspond to negative edges.
The \textit{eigenvalues} of $\dot{G}$ are identified with the eigenvalues of~$A_{\dot{G}}$, and the \textit{spectrum} of $\dot{G}$ is the multiset of its eigenvalues. 
Throughout the paper, by the statement `$\dot G$ has $k$ eigenvalues' we mean that $\dot G$ has exactly $k$ \emph{distinct} eigenvalues.

A \textit{signed $(0, 2)$-graph} is a connected signed graph such that any two vertices have 0 or 2 common neighbours. A triangle-free signed $(0, 2)$-graph is called a \textit{signed rectagraph}. These definitions cover unsigned graphs, as well. Moreover, the entire terminology is transferred from the framework of unsigned graphs~\cite{Bro}. Evidently, $\dot{G}$ is a signed $(0, 2)$-graph (resp.~signed rectagraph) if and only if its underlying graph $G$ is a $(0, 2)$-graph (resp.~rectagraph).

Signed graphs with exactly two eigenvalues have attracted a great deal of attention in recent years. 
They played a central role in the recent breakthrough paper of Huang~\cite{huang}.
Such signed graphs are regular, moreover strongly regular in the sense of~\cite{StaSRSG}, and they are closely related to line systems with maximum number of lines that are either mutually orthogonal or at fixed angle~\cite{Sta2}. 
For various work on signed graphs with just two eigenvalues see \cite{Ho...,McSm,Sta2}. 
See \cite{Bel} for a collection of open problems for signed graphs and see \cite{Gho,Sta3} for some recent results on signed graphs with a symmetric spectrum.

In this paper, we are interested in signed graphs $\dot G$ that have few eigenvalues and also have a symmetric spectrum, i.e., if $\lambda$ is an eigenvalue of $\dot G$ then so is $-\lambda$.
We begin by showing that if $G$ is a $(0,2)$-graph with at most three eigenvalues then $G$ must be a rectagraph.
We then consider signed rectagraphs with spectrum $\{[-\lambda]^m, [\lambda]^m\}$ and we establish some structural and spectral properties of these signed graphs and fully determine all with vertex degree at most $6$. We also provide a partial result related to those with vertex degree $7$ and answer some questions that arise from the previous considerations.
These results are reported in Sections~\ref{sec:prop} and \ref{sec:small}.
In Section~\ref{sec:sub}, we consider induced subgraphs of a signed graph $\dot G$ with spectrum $\{[-\lambda]^m, [\lambda]^m\}$ obtained by deleting one or two vertices and investigate to what extent these subgraphs are determined by their spectra.
Finally, in Section~\ref{sec:few}, we classify signed $(0,2)$-graphs that have symmetric spectra with three eigenvalues and signed $(0,2)$-graphs that have spectrum $\{[-\lambda]^m, [-\mu]^1, [\mu]^1, [\lambda]^m\}$ for {$1 \leqslant \mu < \lambda$}.

\section{Terminology and notation}\label{sec:prel}

In this section, we provide some of the basic terminology and the corresponding notation - some additional notions and known results are given in the forthcoming sections. For undefined notions, we refer the reader to \cite{CRS}.
Throughout, we use $J_n$, $I_n$, $O_n$ to denote the all-$1$, the identity, and the zero matrix, respectively, of order $n$.
The subscript may be omitted when the size of the matrix is clear. 

We say that a signed graph is connected, regular or bipartite if the same holds for its underlying graph. We denote the vertex degree of a regular signed graph by $r$. Throughout the paper, we abbreviate a signed rectagraph whose spectrum consists of precisely two eigenvalues $\pm \lambda$ to SR2SE. We indicate its order and vertex degree by writing $(n, r)$-SR2SE. 
Indeed, $(0,2)$-graphs are known to be regular.

\begin{lemma}[{\cite[Proposition 2.1.2]{mulder}}]
    \label{lem:reg}
    All $(0,2)$-graphs are regular.
\end{lemma}

A \textit{weighing matrix} $M$ of order $n$ and weight $r$ is an $n\times n$ $\{0, \pm 1\}$-matrix satisfying $M^\transpose M=rI$. 
{We write $(n, r)$-weighing matrix to indicate the order $n$ and the weight $w$.}
We say that two rows of a weighing matrix {$M$} intersect in $k$ places if their non-zero entries match in exactly $k$ positions. 
{ Each number $k$ such that there exist two rows of $M$ that intersect in $k$ places is called an \textit{intersection number} of $M$.}

A weighing matrix $Q$ of weight $1$ is also known as $\{0, \pm 1\}$-monomial matrix or a \textit{signed permutation matrix}.
We say that two $n \times n$ $\mathbb Z$-matrices $A$ and $B$ are \emph{switching isomorphic} if there exists a signed permutation matrix $Q$ such that $B = Q^\transpose A Q$.
If two signed graphs $\dot G$ and $\dot H$ have switching isomorphic adjacency matrices then we say that $\dot G$ and $\dot H$ are \emph{switching isomorphic}.
Since switching isomorphism is a similarity relation, it is clear that two switching isomorphic signed graphs must have the same spectrum.
Accordingly, switching isomorphic signed graphs are often identified.


Similarly, weighing matrices $M$ and $N$ are said to be \textit{equivalent} if there are signed permutation matrices $P$ and $Q$ such that $M=PNQ$. Again, equivalent weighing matrices are often identified. 
{ Observe that two inequivalent weighing matrices $M_1$ and $M_2$ of weight $r$ give rise to an infinite family of inequivalent weighing matrices of the same weight. 
Indeed, any direct sum of matrices equivalent to $M_1$ or $M_2$ is also a weighing matrix of weight $r$.}
{ We say that a weighing matrix is \textit{proper} if it is not equivalent to a direct sum of two weighing matrices.
Similarly, the equivalence class of a weighing matrix is called \emph{proper} if it contains a proper weighing matrix.}

If the spectrum of a signed graph $\dot{G}$ consists of two symmetric eigenvalues $\pm\lambda$, then the minimal polynomial of its adjacency matrix is $A^2-\lambda^2 I$, which means that $A_{\dot{G}}^2$ is a weighing matrix of weight $\lambda^2$ and the vertex degree of $\dot{G}$ is $\lambda^2$.
The \textit{negation} $-\dot{G}$ is obtained by reversing the sign of every edge of $\dot{G}$.

\section{Properties of SR2SEs}\label{sec:prop}

We start by proving a result that restricts our considerations from signed $(0, 2)$-graphs to signed rectagraphs.

\begin{theorem}
    \label{thm:trianglefree}
    Let $\dot G$ be a signed $(0,2)$-graph with a symmetric spectrum and at most $3$ eigenvalues.
    Then $\dot G$ is a signed rectagraph.
\end{theorem}
\begin{proof} 
    Let $A$ be the adjacency matrix of $\dot G$.
    Then $A^3-\lambda^2A = O$ for some $\lambda \in \mathbb R$.
    Suppose for a contradiction that $\dot G$ contains a triangle.
    Since $\dot G$ is a signed $(0, 2)$-graph, we have that every edge of a triangle belongs to exactly $2$ triangles.
    It follows that $\dot G$ contains a subgraph, say $\dot H$, in which every edge belongs to exactly $2$ triangles. 
    In other words, the underlying graph $H$ is the graph of one of the three Platonic solids that contain a triangle: the tetrahedron, the octahedron or the icosahedron. 
    Moreover, the octahedron is eliminated immediately since it
contains pairs of vertices with 4 common neighbours.
Since $\dot H$ is a signed $(0, 2)$-graph, every pair of its vertices has no common neighbours in $V(\dot G) \backslash V (\dot H)$.
Note that each vertex of an icosahedron or a tetrahedron is in an odd number of triangles.
Therefore, for each of those vertices, the diagonal entries of $A^3$ cannot be $0$.
This contradicts that fact that $A^3$ is a scalar multiple of $A$.
The conclusion that $H$ is not the tetrahedron nor the icosahedron contradicts the
initial assumption that $\dot G$
contains a triangle, and we are done.
\end{proof}

\begin{figure}[htbp]
    \centering
    \begin{center}
    \begin{tikzpicture}[scale=1.5, auto]
\begin{scope}	
	\foreach \pos/\name in {{(0,0)/a}, {(2,0)/b}, {(1,1.732)/c}, {(1,0.607)/d}}
		\node[vertex] (\name) at \pos {}; 
	\foreach \edgetype/\source/ \dest in {pedge/c/a, pedge/b/d, pedge/b/a, nedge/c/d, pedge/a/d, pedge/c/b}
	\path[\edgetype] (\source) -- (\dest);
\end{scope}
\end{tikzpicture}
\end{center}
    \caption{The signed tetrahedron $\mathcal T$.}
    \label{fig:T}
\end{figure}
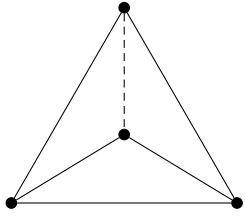

Note that there does exist a signed $(0,2)$-graph with a symmetric spectrum of four eigenvalues that is not a {signed} rectagraph.
Indeed, the signed tetrahedron $\mathcal T$ in Figure~\ref{fig:T} has spectrum $\{ [\pm 1]^1, [\pm \sqrt{5}]^1\}$.

Below, we refer to the $(r+1)\times \big({r+1\choose 2}+1+k\big)$ matrix of \eqref{schem} in which $+$ and $-$ denote 1 and $-1$, respectively, while every unspecified entry is zero. To explain its structure, the top-left submatrix corresponds to the star $K_{1, r}$ and it is followed by $(r-i+1)\times (r-i)$ submatrices (for $1\leqslant i\leqslant r-1$) in which the first row is all-1 and the remainder of the matrix is the negation of the corresponding identity matrix. The $(r+1)\times k$ submatrix at the end is all-0.
\begin{align}\label{schem}&\begin{bmatrix*} &+&+&+&+&\cdots& +&  & &&&&&&&&&&&&\\
+&&& & && &+&+&+&\cdots&+&&&&&&&&&\\
+&&& & && &-& &&& &+&+&\cdots&+&&&&&\\
+&&&&&& & &-&&& &-& &&&\cdots&&&&\\
+&&&&&& & &&-&& &&-&&&\ddots&&&&\\
\vdots&&&&&&&&&&\ddots&&&&\ddots&&&+&&&\\
+&&& & && & & &&&-&&&&-&&-&&&
\end{bmatrix*} \\
&\,~~\underbrace{\phantom{~~~~~~~~~~~~~~~~~~~~~~~~~~~~~~~~~}}_{r+1}~~ \underbrace{~~~~~~~~~~~~~~~~~~~~~~~~}_{r-1}~~\underbrace{~~~~~~~~~~~~~~~~~~~}_{r-2}~\ldots~\,\hspace{0.2mm}{\tiny\underbrace{}_{1}}~~\underbrace{~~~~~~}_{k} \nonumber
\end{align}

\smallskip

We will also deal with the following $r\times \big({r \choose 2}+1+k\big)$ matrix -- a slight modification of the previous one. 

\begin{align}\label{schemb}&\begin{bmatrix*} +&+&+&+&+&\cdots& +&  & &&&&&&&&&&&&\\
+&-&& & && &+&+&+&\cdots&+&&&&&&&&&\\
+&&-& & && &-& &&& &+&+&\cdots&+&&&&&\\
+&&&-&&& & &-&&& &-& &&&\cdots&&&&\\
+&&&&-&& & &&-&& &&-&&&\ddots&&&&\\
\vdots&&&&&\ddots&&&&&\ddots&&&&\ddots&&&+&&&\\
+&&& & &&- & & &&&-&&&&-&&-&&&
\end{bmatrix*} \\
&\,~~\underbrace{\phantom{~~~~~~~~~~~~~~~~~~~~~~~~~~~~~~~~~}}_{r}~~ \underbrace{~~~~~~~~~~~~~~~~~~~~~~~~}_{r-2}~~\underbrace{~~~~~~~~~~~~~~~~~~~}_{r-3}~\ldots~\,\hspace{0.2mm}{\tiny\underbrace{}_{1}}~~\underbrace{~~~~~~}_{k} \nonumber
\end{align}

\begin{lemma}\label{lem:schem}
	The switching isomorphism class of every SR2SE contains a signed graph such that the first $r+1$ rows of its adjacency matrix are given by \eqref{schem}.
\end{lemma}

\begin{proof} We apply an appropriate labelling of the vertices of $\dot{G}$ and a sequence of switching isomorphisms.  We first label the vertices of $\dot{G}$ is such a way that 1 corresponds to an arbitrary vertex, and the vertices $2, 3, \ldots, r+1$ are its neighbours. Since $\dot{G}$ is triangle-free the vertices $2, 3, \ldots, r+1$ are mutually non-adjacent. Since $\dot{G}$ is a signed $(0, 2)$-graph, each of vertices $3, 4, \ldots, r+1$ share exactly 2 common neighbours with 2: 1 and another one. We enumerate the additional neighbours by $r+2$, $r+3, \ldots, 2r$, and proceed by enumerating the common neighbours of 3 and each of $4, 5, \ldots, r+1$ by $2r+1, 2r+2, \ldots, 3r-2$, and so on. After completing the enumeration of all vertices that have exactly two common neighbours with the vertex 1, we continue by making a switch at appropriate vertices such that in the resulting signed graph all edges incident with 1 and 2 are positive and, for $3\leqslant i\leqslant r+1$  an edge $ij$ is positive if and only if  there is no edge $kj$ for $k<i$. In this way we arrive at the desired signed graph.
\end{proof}

We continue with a similar result concerning weighing matrices.

\begin{lemma}
\label{lem:schemWeigh}
	Every weighing matrix $W$ of weight $r$ and intersection numbers $0$ and $2$ is equivalent to a weighing matrix whose first $r$ rows are given by~\eqref{schemb}.
\end{lemma}

\begin{proof} Since $W^2=rI$, we have that every pair of rows that intersect at 2 places, in one place intersect at elements of the same sign and in the other at elements that differ in sign. In addition, we may assume that the first $r$ rows are as in \eqref{schemb}, and the result follows.
\end{proof}

An $n$-vertex, $r$-regular graph $G$ is called \emph{strongly regular} with parameters $(n,r,a,c)$ if every pair of adjacent vertices has $a$ common neighbours and every pair of non-adjacent vertices has $c$ common neighbours.
Now we list some basic properties of SR2SEs.

\begin{theorem}\label{the:1-4}
	If $\dot{G}$ is an $(n,r)$-SR2SE, then:
	\begin{itemize}
		\item[(i)] $n\geqslant {r+1\choose 2}+1$. In case of equality, the underlying graph $G$ is strongly regular with parameters $(n, r, 0, 2)$;
		
		\item[(ii)] $\dot{G}$ contains exactly $\frac{n}{4}{r\choose 2}$ quadrangles;
		
		\item[(iii)]  If $n\equiv 2~(\md 4)$, then (a) $r$ is the sum of two squares and (b) $r(r-1)\equiv 0~(\md 4)$;
		
		\item[(iv)] The sum of the cubes (resp.~fourth powers) of the eigenvalues of $G$ is equal to zero (resp.~$nr(3r-2)$).
	\end{itemize}
\end{theorem}

\begin{proof} (i): By Lemma~\ref{lem:schem}, we may assume that the first $r+1$ rows of $A_{\dot{G}}$ are as in~\eqref{schem}.
Thus $n\geqslant r+1+ {r\choose 2}+k$, that is,  $n\geqslant {r+1\choose 2}+1$. 
In case of equality, we have $k=0$ which means that every pair of non-adjacent vertices of $G$ has 2 common neighbours. 
Since $G$ is triangle-free, we find that $G$ is strongly regular with parameters $(n,r,0,2)$.
	
	(ii): Observe that, according to \eqref{schem}, the first vertex has ${r\choose 2}$ non-neighbours such that it shares 2 common neighbours with each of them. This gives ${r\choose 2}$ quadrangles containing the first vertex, i.e., 	$\frac{n}{4}{r\choose 2}$ quadrangles in total.
	
	(iii):
    For $n\equiv 2~(\md 4)$, (a) is a necessary condition for the existence of the adjacency matrix of $\dot{G}$, which can be found in \cite[p. 300]{Seb}, while (b) follows from the item (ii).
	
	(iv): Since $G$ is triangle-free, the trace of $A_G^3$ is zero, which gives the sum of the third degrees of its eigenvalues. Further, every diagonal entry of $A_G^4$ is equal to the number of closed walks of length 4 that start and terminate at the corresponding vertex. 
	This number is $r^2+r(r-1)+2{r\choose 2}$.
	Indeed, if the corresponding vertex is $u$ then the first term counts $4$-walks of the form $(u,v,u,w,u)$ where $v$ and $w$ can be equal, the second counts those of the form $(u,v,w,v,u)$ where $u$ and $w$ are distinct, and the third counts those of the form $(u,v,w,x,u)$ where all four vertices are distinct. By multiplying by $n$ we obtain the result.
\end{proof}

We can say more in the case of bipartite SR2SEs. We note in passing that, according to \cite{Bro}, they may be seen as signed semibiplanes with eigenvalues $\pm \lambda$.
First we show a relation between $(2n, r)$-SR2SEs and $(n,r)$-weighing matrices.

\begin{theorem}\label{the:bip1}
 Every proper { equivalence class of} $(n, r)$-weighing matrices with intersection numbers $0$ and $2$ uniquely determines a switching isomorphism class of bipartite $(2n, r)$-SR2SEs, and vice versa.
\end{theorem}
\begin{proof}
     If $M, N$ are proper equivalent $(n, r)$-weighing matrices, then there exist signed permutation matrices $P, Q$, such that $M=PNQ$. 
     We set
	\begin{equation*}\label{A_G}A_{\dot{G}}=\begin{bmatrix}
	O&M^\transpose \\ M& O
	\end{bmatrix}~~\text{and}~~A_{\dot{H}}=\begin{bmatrix}
	O&N^\transpose \\ N& O
	\end{bmatrix}.\end{equation*}
	Evidently, $\dot{G}, \dot{H}$ are bipartite $(0, 2)$-signed graphs with eigenvalues $\pm\sqrt{r}$ (the latter is because $A^2_{\dot{G}}=A^2_{\dot{H}}=rI$, which follows easily). Further,
	$$A_{\dot{G}}=\begin{bmatrix}
	Q&O\\O&P^\transpose
	\end{bmatrix}^{-1}\begin{bmatrix}
	O&N^\transpose\\N&O
	\end{bmatrix}\begin{bmatrix}
	Q&O\\O&P^\transpose
	\end{bmatrix}$$
	implies that $\dot{G}$ and $\dot{H}$ are switching isomorphic. Finally, if the matrix equivalence class does not contain a matrix of the form
	$$\begin{bmatrix}
	M_1&O\\O&M_2
	\end{bmatrix},$$
	the corresponding signed graphs are connected, and we are done. 
	
The converse can be shown in a similar manner.
\end{proof}

A \textit{symmetric balanced incomplete block design} (a \textit{symmetric BIBD}) with parameters $(n, r, l)$ is an arrangement of $n$ points into the $n$ blocks of size $r$ in such a way that every point is contained in $r$ blocks and every pair of points occurs together in $l$ blocks. The vertices of the \textit{incidence graph} of a symmetric BIBD are indexed by the points and the blocks, and two vertices are adjacent if and only if one is indexed by a point and other by a block containing that point. This graph is bipartite.

\begin{theorem}\label{the:bip2} Let $\dot{G}$ be a bipartite $(2n,r)$-SR2SE.
	Then
	\begin{itemize}
%
		
		\item[(i)] $n\geqslant {r\choose 2}+1$. In case of equality, the underlying graph $G$ is the incidence graph of a symmetric BIBD with parameters $(n, r, 2)$;
		
		\item[(ii)] If $n$ is odd, then (a) $r$ is a square, (b) $n\leqslant (n-r)^2+n-r+1$ and (c) $r(r-1)\equiv 0~(\md 4)$. If $n\equiv 2~(\md 4)$, then  $r$ is the sum of two squares. 
	\end{itemize}
\end{theorem}

\begin{proof}
By Theorem~\ref{the:bip1}, we can write 
	\begin{equation*}A_{\dot{G}}=\begin{bmatrix}
	O&B^\transpose \\ B& O
	\end{bmatrix},
	\end{equation*}
	where $B$ is a proper weighing matrix.
	
	(i): Using Lemma~\ref{lem:schemWeigh}, we may assume that the first $r$ rows of $B$ are as in~\eqref{schemb}.
Thus $n\geqslant r+ {r - 1\choose 2}+k$, that is,  $n\geqslant {r\choose 2}+1$. 
In case of equality, i.e., $k=0$, we have that every two vertices of the first colour class have 2 common neighbours in the second, and vice versa. Taking into account $n$ and $r$, we get that $G$ is the incidence graph of a symmetric BIBD with parameters $(n, r, 2)$, as required.
	
	(ii): 
    The claims (a), (b) are necessary conditions from \cite[p. 300]{Seb} for the existence of the proper weighing matrix $B$. Further,
by Theorem~\ref{the:1-4}(ii) $\dot{G}$ has $\frac{n}{2}{n\choose 4}$ quadrangles.
Thus, we arrive at (c). 
For $n\equiv 2~(\md 4)$, the result again follows from
the necessary condition from \cite[p. 300]{Seb} for the existence of the proper weighing matrix $B$. 
\end{proof}

At this point we need the following compositions of signed graphs. We recall that the Cartesian product $\dot{G}\times K_2$ of a signed graph  $\dot{G}$ and $K_2$ is obtained by taking two copies of $\dot{G}$ and inserting a positive edge between two copies of every vertex. Here we need a modification of this product denoted by $\dot{G}\ltimes K_2$ and obtained from $\dot{G}\times K_2$ by reversing the sign of every edge in exactly one copy of $\dot{G}$.  

\begin{lemma}\label{lem:biTensor} 
    Let $\dot G$ be a signed graph with $\det \left (xI - A_{\dot G}\right ) = x^{m_0}\prod_{i=1}^l(x^2-\lambda_i)^{m_i}$.
    Then $\det\left (xI-A_{\dot{G}\ltimes K_2}\right ) = (x^2-1)^{m_0}\prod_{i=1}^l(x^2-\lambda_i-1)^{2m_i}$.
\end{lemma} 
\begin{proof}
    The adjacency matrix of $\dot{G}\ltimes K_2$ is
	$$A_{\dot{G}\ltimes K_2}=\begin{bmatrix}
		A_{\dot{G}}& I_n\\ I_n&-A_{\dot{G}}
	\end{bmatrix}.$$
	Since all blocks commute with each other, we have \begin{align*}
	\det\left (xI-A_{\dot{G}\ltimes K_2}\right )&=\det\left((xI_n-A_{\dot{G}})(xI_n+A_{\dot{G}})-I_n\right)=\det\left((x^2-1)I_n-A_{\dot{G}}^2\right).
	\end{align*}
	The lemma follows since  $\det \left ( xI - A^2_{\dot G}\right ) = x^{m_0}\prod_{i=1}^l(x-\lambda_i)^{2m_i}$.
\end{proof}

The \textit{bipartite double} $\bd(\dot{G})$ of a signed graph $\dot{G}$ with vertex set $\{i_1, i_2, \ldots, i_n\}$ has the vertex set $\{i_11, i_12, i_21, i_22, \ldots, i_n1, i_n2\}$ and there is a positive (resp.~negative) edge between $i_uj$ and $i_vk$ if and only if there is a positive (negative) edge between $i_u$ and $i_v$ and $j\neq k$. The adjacency matrix of $\bd(\dot{G})$ is the Kronecker product $A_{\dot{G}}\otimes A_{K_2}$. 
The signed graph $\bd(\dot{G})$ is bipartite and it is connected if and only if $\dot{G}$ is non-bipartite.

\begin{theorem}\label{the:const} If $\dot{G}$ is an $(n, r)$-SR2SE, then $\dot{G}\ltimes K_2$ is a $(2n, r+1)$-SR2SE. If $\dot{G}$ is non-bipartite, then $\bd(\dot{G})$ is a $(2n, r)$-SR2SE; otherwise, $\bd(\dot{G})$ consists of two disjoint copies of $\dot{G}$.
\end{theorem} 

\begin{proof}  It follows from definition that $\dot{G}\ltimes K_2$ is an $(r+1)$-regular signed rectagraph, and if $\dot{G}$ is non-bipartite then $\bd(\dot{G})$ is an $r$-regular signed rectagraph, while for otherwise it consists of two disjoint copies of $\dot{G}$. Therefore, it remains to show that both compositions have 2 symmetric eigenvalues. 
Since $\det \left (xI - A_{\dot G}\right ) = (x^2-r)^{n/2}$, by Lemma~\ref{lem:biTensor} the signed graph $\dot{G}\ltimes K_2$ has exactly two eigenvalues $\pm \sqrt{r+1}$.
	
	Since $A_{\bd(\dot{G})}=A_{\dot{G}}\otimes A_{K_2}$, the eigenvalues of $A_{\bd(\dot{G})}$ are the eigenvalues of $A_{\dot{G}}$ multiplied by $\pm 1$, and the result follows.
\end{proof}

We proceed with a useful corollary.

\begin{corollary}\label{cor:uniq} The following statements hold:
	
	\begin{itemize}
		\item[(i)] If $\dot{G}$ is a unique SR2SE with underlying graph $G$, then $\dot{G}\ltimes K_2$ is a unique SR2SE with underlying graph $G\times K_2$.
		
		\item[(ii)] If $G\times K_2$ is the underlying graph of an SR2SE, so is $G$. 
		\end{itemize} 
\end{corollary}

\begin{proof}(i): Assume that $\dot{G}$ is a unique SR2SE with underlying graph $G$. By Theorem~\ref{the:const}, $\dot{G}\ltimes K_2$ is an SR2SE, so it remains to show its uniqueness. Assume to the contrary and let $\dot{H}$ be another SR2SE with the same underlying graph. Observe that the subconstituents of $\dot{H}$ whose underlying graphs are $G$ and its copy are SR2SEs, as well. According to the starting assumption, both are switching isomorphic to $\dot{G}$. Without loss of generality, we may take that all the edges between them are positive, since for otherwise we can make an appropriate switching. 
	
	 Let $u, v$ be adjacent vertices in one copy and $u', v'$ their copies in the other. Since these four vertices make a negative quadrangle and the edges $uu'$ and $vv'$ are positive, we get that $uv$ and $u'v'$ differ in sign. In other words, one copy is the negation of the other, so $\dot{H}$ is switching isomorphic to $\dot{G}\ltimes K_2$, and we are done.
	 
	 \smallskip 
	 
	 (ii): As before, if $\dot{H}$ is an SR2SE with underlying graph $G\times K_2$, then the same holds for its subconstituent that arises from one copy of $G$, i.e., whose underlying graph is $G$, and we are done.
\end{proof}

Theorem~\ref{the:const} allows us to construct an infinite family of SR2SEs. 

\begin{corollary}\label{cor:Gi} If 
	$$\left\{\begin{array}{rl}\dot{G}_1 =&\hspace{-2mm} K_2,\\ \dot{G}_{r+1} =&\hspace{-2mm}\dot{G}_r\ltimes K_2,~\text{for}~r\geqslant 1,\end{array}\right.$$
	then $\dot{G}_r$, for $r\geqslant 1$, is a signed rectagraph with eigenvalues $\pm \sqrt{r}$.
 \end{corollary}

\begin{proof}
	The result follows from Theorem~\ref{the:const} as the eigenvalues of $K_2$ are $\pm 1$.
\end{proof}

The reader will recognize that the signed graph $G_r$ from the previous corollary is the $r$-dimensional signed cube with negative quadrangles. According to \cite{StaFil}, up to switching isomorphism, this is the unique $r$-dimensional signed cube with negative quadrangles.

\section{SR2SEs of small order}\label{sec:small}

Here we determine all SR2SEs whose vertex degree $r$ is at most 6. We also consider some particular cases for $r=7$. Since the underlying graph of every signed rectagraph is an unsigned rectagraph, throughout this section we refer to Brouwer's \cite{Bro} where all $(0,2)$-graphs with $r\leqslant 7$ are listed and those that are rectagraphs are indicated. All of them can be found under the unchanged enumeration in the extended lists of $(0,2)$-graphs with $r\leqslant 8$ located at the web page \cite{BroS}. The results of this section are summarized in Table~\ref{tab:tab}. The first column of the table contains the SR2SE's identification, the second contains its number of vertices, the third  contains the identification of the underlying graph taken from \cite{BroS} and eventual additional description (where, for example $2^r$ denotes the $r$-dimensional cube), the fourth column contains a closer description of the SR2SE and the fifth column refers to the corresponding result.

\subsection{$r\leqslant 5$}\label{ssec:5}

First, the SR2SEs of degree at most 2 are $\dot{G}_1$ and $\dot{G}_2$ of Corollary~\ref{cor:Gi} which is verified easily, and those with $r\in\{3, 4\}$ can be identified among signed graphs of specified degree and exactly 2 eigenvalues. All of them can be found in \cite{Ho...,Sta2}, and they can also be deduced from the results of \cite{McSm}. It occurs that exactly 3 of them are SR2SEs -- $\dot{G}_3$ (of degree 3), the signed graph of Figure~\ref{fig:small}(a) and $\dot{G}_4$. 

We denote the obtained SR2SEs by R1.1, R2.1, R3.1, R4.1 and R4.2, respectively.

Next, SR2SEs with $r=5$ can be deduced from the incomplete list of signed graphs with eigenvalues $\pm\sqrt{5}$ of \cite{Stas5}. Namely, this list contains all signed rectagraphs and among them there are exactly 4 SR2SEs. The three of them are bipartite and they are obtained as in Theorem~\ref{the:bip1} on the basis of weighing matrices of order 12, 14 and 16, respectively. These matrices can be found in \cite{wm,Stas5} (in both references they are denoted by $W_{12, 5}$, $W_{14, 5}$ and $D(16, 5)$). The latter one is the smallest non-bipartite SR2SE illustrated in Figure~\ref{fig:small}(b).
Its underlying graph is the Clebsch graph, also known as the folded $4$-dimensional cube. (We recall that the folded $n$-dimensional cube is obtained by inserting an edge between every pair of vertices at maximal distance in the $n$-dimensional cube.) 

We denote the obtained SR2SEs with $r=5$ by R5.1--R5.4, respectively.

	\begin{figure}
	\centering
	\includegraphics[width=130mm,angle=0]{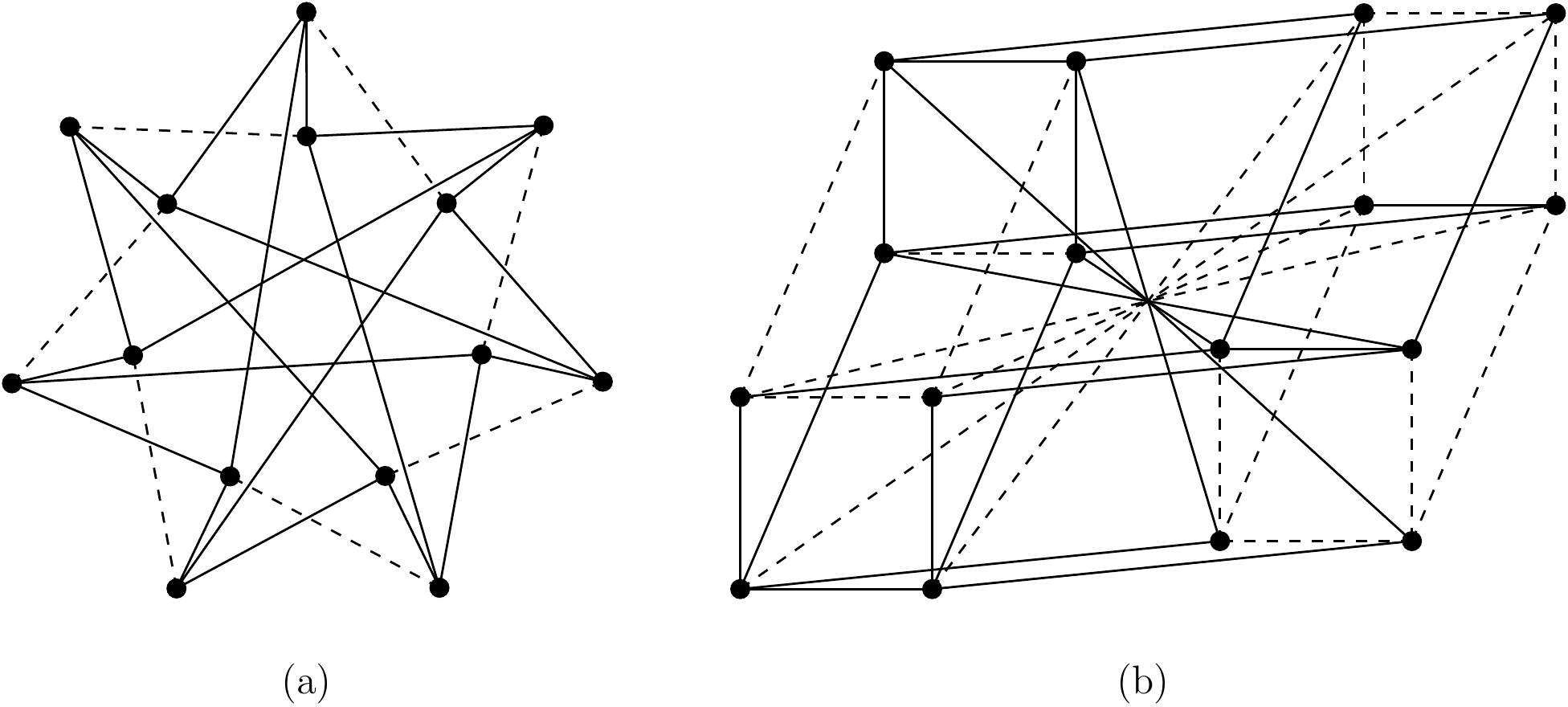}
	\caption{$(14, 4)$-SR2SE R4.1 and $(16, 5)$-SR2SE R5.4.}\label{fig:small}
\end{figure}

Observe that the SR2SE R5.4 of Figure~\ref{fig:small} attains the lower bound of Theorem~\ref{the:1-4}(i). This leads to the following question.

\begin{question}
\label{q1}
Is it true that every strongly regular graph with parameters $(n, r, 0, 2)$ is the underlying graph of an SR2SE?
\end{question}

 Since the underlying graph of R5.4 is also a folded cube, we may ask the following.
 
\begin{question}
\label{q2}
Is it true that every folded $r$-dimensional cube is the underlying graph of an SR2SE?
\end{question}

For Question~\ref{q2} we should restrict ourselves to $r\geqslant 4$, since for $r=3$ the corresponding cube is not a rectagraph. However, we will see in the forthcoming Subsections~\ref{ssec:6} and \ref{ssec:G} that the answer to both Question~\ref{q1} and Question~\ref{q2} is negative.

\subsection{$r=6$}\label{ssec:6}

We first consider the bipartite SR2SEs, i.e., those whose underlying graphs are enumerated by 6.1--6.13 in \cite{BroS}. 

Since 6.1--6.3 have 32 vertices, by Theorem~\ref{the:bip1}, the existence of an SR2SE whose underlying graph is any of them would imply the existence of a proper { equivalence} class of $(16, 6)$-weighing matrices with intersection numbers 0 and 2. According to \cite{wm}, there are exactly 30 classes of $(16, 6)$-weighing matrices. We check their representatives, and none of them has the required intersection numbers. So, there are no solutions in this case. Observing that 6.3 is the folded 5-dimensional cube, we arrive at a negative answer to Question~\ref{q2}.

The underlying graphs 6.4--6.6 are considered in a similar way. These graphs have 36 vertices, and by Theorem~\ref{the:bip2}(ii) there are no corresponding SR2SEs since 6 is not the sum of two squares. In the same way we eliminate 6.9.

We now consider 6.7, 6.8, 6.10 and 6.11. Graphs 6.7 and 6.8 have 40 vertices, each, so we consider proper inequivalent $(20, 6)$-weighing matrices. There are 49 such matrices and exactly two of them have intersection numbers 0 and 2. These are the 46th and the 47th matrix in the corresponding list of \cite{Mun}. According to Theorem~\ref{the:bip1}, each of them gives the unique SR2SE with the corresponding parameters. (The 46th matrix corresponds to 6.8 and the 47th corresponds to 6.7.)
Similarly, 6.10 and 6.11 have 48 vertices. There are 190 proper inequivalent $(24, 6)$-weighing matrices and exactly two of them have intersection numbers 0 and 2: the 2nd and the 18th on \cite{Mun}. Each of them gives the unique SR2SE with the corresponding parameters -- the first corresponds to 6.11 and the second to 6.10.

The graph 6.12 is 5.3$\times K_2$. Since 5.3 is the underlying graph of R5.2, by Corollary~\ref{cor:uniq}(i), we get that 6.12 is the underlying graph of the unique SR2SE isomorphic to R5.2$\ltimes K_2$. Finally, 6.13 is the 6-dimensional cube, and so it is the underlying graph of $\dot{G}_5\ltimes K_2$ (with the notation of Corollary~\ref{cor:Gi}). The uniqueness of $\dot{G}_5\ltimes K_2$ follows by Corollary~\ref{cor:uniq}, since $\dot{G}_4$ is unique.

We proceed with non-bipartite SR2SEs. According to \cite{BroS}, there are exactly two candidates for their underlying graphs, denoted by N6.6 and N6.9. First, N6.6 is eliminated by computer search. Namely, we label its vertices in such a way that the first 7 rows of the adjacency matrix of a putative SR2SE are as in \eqref{schem}. Then consider all possible signatures for the remaining edges and arrive at no solution. Next, as before, we get that N6.9 is the underlying graph of the unique solution R5.2$\ltimes K_2$. 

We denote the seven SR2SEs obtained in this subsection by R6.1--R6.7, respectively.

\subsection{$r=7$ -- a partial result.}

Here we use Corollary~\ref{cor:uniq}(i) to confirm the existence and the uniqueness and Corollary~\ref{cor:uniq}(ii) to prove the non-existence is some particular cases.

There are 40 bipartite rectagraphs with $r=7$. Among them 7.34, 7.35, 7.37--7.40 have the form $G\times K_2$, where $G$ is the underlying graph of a confirmed SR2SE. Therefore, each of them is the underlying graph of an SR2SE whose uniqueness follows from the uniqueness of $G$. The details are given in Table~\ref{tab:tab}. 

Similarly, 7.20, 7.21, 7.23, 7.30--7.32 and 7.36 have the form $G\times K_2$, where $G$ is not the underlying graph of some SR2SE. By Corollary~\ref{cor:uniq}(ii), there is no SR2SE for any of these underlying graphs.
The remaining bipartite rectagraphs remain unresolved.

Again, according to \cite{BroS}, there are 5 non-bipartite rectagraphs with $r=7$, denoted by N7.44, N7.45, N7.51--N7.53. There is no SR2SE with underlying graph N7.45, since N7.45=N6.6$\times K_2$, but N6.6 is not the underlying graph of an SR2SE, as we already showed. There is a unique SR2SE with underlying graph N7.52 isomorphic to R6.7$\ltimes K_2$ (since N7.52=N$6.9\times K_2$).
The remaining non-bipartite rectagraphs remain unresolved.

We denote the seven SR2SEs obtained in this subsection by R7.1--R7.7, respectively.

\begin{table}[h!t]
	\begin{center}
		\caption{SR2SEs with $r\leqslant 7$.}
		\label{tab:tab}
		\begin{tabular}{ccccl} 
			$\#$ & $n$ & Underlying graph&Description&\multicolumn{1}{c}{Reference}\\ 
			\hline
			R1.1&2&1.1, $2^1$&$\dot{G}_1$&Cor.~\ref{cor:Gi} or \cite{Ho...,Sta2}\\
			R2.1&4&2.1, $2^2$&$\dot{G}_2$&Cor.~\ref{cor:Gi} or \cite{Ho...,Sta2}\\
			R3.1&8&3.1, $2^3$&$\dot{G}_3$&Cor.~\ref{cor:Gi} or \cite{Ho...,Sta2}\\
			R4.1&14&4.1&Figure~\ref{fig:small}(a)&\cite{Ho...,Sta2}\\
				R4.2&16&4.2, $2^4$&$\dot{G}_4$&Cor.~\ref{cor:Gi} or \cite{Ho...,Sta2}\\
			R5.1&24&5.2&$W_{12,5}$&Thm.~\ref{the:bip1} or \cite{Ho...,Sta2} \\
			R5.2&28&5.3&$W_{14, 5}$&Thm.~\ref{the:bip1} or \cite{Ho...,Sta2} \\
			R5.3&32&5.4&$D(16, 5)$&Thm.~\ref{the:bip1} or \cite{Ho...,Sta2} \\
			R5.4&16&N5.2, Clebsch graph&Figure~\ref{fig:small}(b), non-bip.&\cite{Ho...,Sta2}\\
			R6.1&40&6.8&$W_{20, 6}$&Thm.~\ref{the:bip1} and \cite{Mun}\\
				R6.2&40&6.7&$W_{20, 6}$&Thm.~\ref{the:bip1} and \cite{Mun}\\
			R6.3&48&6.11&$W_{24,6}$&Thm.~\ref{the:bip1} and \cite{Mun}\\
			R6.4&48&6.10&$W_{24,6}$&Thm.~\ref{the:bip1} and \cite{Mun}\\
			R6.5&56&6.12, 5.3$\times K_2$&R5.2$\ltimes K_2$&Thm.~\ref{the:const} and Cor.~\ref{cor:uniq}(i)\\
			R6.6&64&6.13, $2^6$&$\dot{G}_6$&Cor.~\ref{cor:Gi} and Cor.~\ref{cor:uniq}(i)\\
			R6.7&32&N6.9, N5.2$\times K_2$&R5.4$\ltimes K_2$, non-bip.&Thm.~\ref{the:const} and Cor.~\ref{cor:uniq}(i)\\
			R7.1&80&7.34, 6.7$\times K_2$&R6.2$\ltimes K_2$&Thm.~\ref{the:const} and Cor.~\ref{cor:uniq}(i)\\
			R7.2&80&7.35, 6.8$\times K_2$&R6.1$\ltimes K_2$&Thm.~\ref{the:const} and Cor.~\ref{cor:uniq}(i)\\
			R7.3&96&7.37, 6.10$\times K_2$&R6.4$\ltimes K_2$&Thm.~\ref{the:const} and Cor.~\ref{cor:uniq}(i)\\
			R7.4&96&7.38, 6.11$\times K_2$&R6.3$\ltimes K_2$&Thm.~\ref{the:const} and Cor.~\ref{cor:uniq}(i)\\
			R7.5&112&7.39, 6.12$\times K_2$&R6.5$\ltimes K_2$&Thm.~\ref{the:const} and Cor.~\ref{cor:uniq}(i)\\
			R7.6&128&7.40, $2^7$&$\dot{G}_7$&Cor.~\ref{cor:Gi} and Cor.~\ref{cor:uniq}(i)\\
			R7.7&64&N7.52, N6.9$\times K_2$&R6.7$\ltimes K_2$, non-bip.&Thm.~\ref{the:const} and Cor.~\ref{cor:uniq}(i)\\
		\end{tabular}
	\end{center}
\end{table}

\subsection{Signing the Gewirtz graph} \label{ssec:G}

The Gewirtz graph is the unique strongly regular graph with parameters $(56, 10, 0, 2)$. Following the method applied in the case of N6.6, we consider the signing in this graph in order to show that there is no corresponding SR2SE. After an appropriate labelling and reversing the sign of a fixed set of edges, we arrive at the signed graph such that the first 11 rows of its adjacency matrix are as in \eqref{schem}. The first 11 columns are transpose of these rows. The remaining $45\times 45$ submatrix has fixed zero entries, so it remains to consider all the possible  $\{\pm 1\}$-substitutions for the non-zero positions of this matrix.
This procedure is performed very quickly under the condition that a putative signed graph is an SR2SE. Namely, we get just 8 possibilities for the first row of the submatrix that are compatible with the first 11 rows. By proceeding with the subsequent rows, the computer quickly arrives at no possible solution. So, there is no  SR2SE whose underlying graph is the Gewirtz graph. 
In this way we obtain a negative answer to Question~\ref{q1}. The next possibility for an SR2SE that attains the bound of Theorem~\ref{the:1-4}(i) is a putative signed graph whose underlying graph is also a putative strongly regular graph with parameters $(352, 26, 0, 2)$.

\section{Induced subgraphs of signed graphs with spectrum $\{[-\lambda]^m,[\lambda]^m\}$}
\label{sec:sub}

In this section, we consider the induced subgraphs of a signed graph with spectrum $\{[-\lambda]^m,[\lambda]^m\}$ obtained by deleting one or two vertices.
We are interested in to what extent large subgraphs of such signed graphs are determined by their spectrum.
See \cite{Gre} for a similar consideration for conference matrices.

We will require the use of \emph{eigenvalue interlacing}, which we record as a lemma.

\begin{lemma}[{\cite[Chapter 9]{GoRo}}]
    Let $\dot G$ be a signed graph and let $\dot H$ be an induced subgraph of $\dot G$.
    Suppose $\dot G$ and $\dot H$ have eigenvalues 
    \[
    \lambda_1 \geqslant \lambda_2 \geqslant \dots \geqslant \lambda_n \quad \text{ and } \quad  \mu_1 \geqslant \mu_2 \geqslant \dots \geqslant \mu_m,
    \]
    respectively.
    Then for each $i = 1,2,\dots, m$, we have $ \lambda_{i} \geqslant \mu_i \geqslant \lambda_{n-m+i}$.
\end{lemma}

Let $\dot G$ be a signed graph having spectrum $\{[-\lambda]^m,[\lambda]^m\}$.
Then since $A_{\dot G}^2 = \lambda^2I$, each vertex of $\dot G$ has degree $\lambda^2$.
Now delete a vertex from $\dot G$ to obtain the induced subgraph $\dot H$.
Using eigenvalue interlacing and that the sum of its eigenvalues is $0$, we know that $\dot H$ has spectrum $\{[-\lambda]^{m-1},[0]^1,[\lambda]^{m-1}\}$.
Furthermore, each vertex of $\dot H$ has degree $\lambda^2$ or $\lambda^2-1$.
Therefore, the existence of a signed graph with spectrum $\{[-\lambda]^m,[\lambda]^m\}$ implies the existence of a signed graph with spectrum $\{[-\lambda]^{m-1},[0]^1,[\lambda]^{m-1}\}$ whose vertices have degrees $\lambda^2$ and $\lambda^2-1$.
Now we show that the converse is also true.

\begin{theorem}
\label{thm:reconst1}
    Let $\dot G$ be a signed graph with spectrum $\{[-\lambda]^m,[0]^1,[\lambda]^m\}$.
    Suppose $G$ has vertex degrees in $\{\lambda^2,\lambda^2-1\}$.
    Then $\dot G$ can be extended by a vertex to obtain a signed graph with spectrum $\{[-\lambda]^{m+1},[\lambda]^{m+1}\}$.
\end{theorem}
\begin{proof}
    The matrix $M := \lambda^2I-A_{\dot G}^2$ is a rank-1 matrix with all diagonal entries equal to $0$ or $1$. 
    Thus $M = \mathbf x \mathbf x^\transpose$, where $\mathbf x$ is a $\{0,\pm 1\}$-vector.
    Now form the matrix
    \[
    B = \begin{bmatrix} 0 & \mathbf x^\transpose \\
                        \mathbf x & A_{\dot G}
    \end{bmatrix}.
    \]
    Observe that $B$ is the adjacency matrix of a signed graph $\dot H$ that has spectrum equal to $\{[-\lambda]^{m+1},[\lambda]^{m+1}\}$, as required.
    
    Since $B^2 = \lambda^2I$, each vertex of $\dot H$ has degree equal to $\lambda^2$.
    Therefore, each vertex of $\dot G$ has degree equal to $\lambda^2$ or $\lambda^2-1$.
\end{proof}

Note that, in Theorem~\ref{thm:reconst1}, we need a degree condition.
Indeed, we can obtain a signed graph $\dot H$ with spectrum $\{[-\lambda]^m,[0]^1,[\lambda]^m\}$ by taking the union of an isolated vertex with a signed graph with spectrum $\{[-\lambda]^m,[\lambda]^m\}$.
But it is straightforward to see that $\dot H$ is not an induced subgraph of any signed graph with spectrum $\{[-\lambda]^{m+1},[\lambda]^{m+1}\}$.



Now suppose we delete two vertices from $\dot G$.
This time the resulting graph will have one of the two possible spectra: $\{[-\lambda]^{m-2},[0]^2,[\lambda]^{m-2}\}$ or $\{[-\lambda]^{m-2},[-1]^1,[1]^1,[\lambda]^{m-2}\}$.
To see that these are the only two possible spectra, let $\dot H$ be the signed graph obtained by deleting two vertices $x$ and $y$ from $\dot G$.
By interlacing and using the fact that the trace of $A_{\dot H}$ is $0$, the spectrum of $\dot H$ has the form $\{[\lambda]^{m-2}, [-\mu]^1, [\mu]^1, [\lambda]^{m-2}\}$ for some $\mu \in \mathbb R$.
Now the trace of $A_{\dot H}^2$ equals either $\tr(A_{\dot G}^2) - 4\lambda^2$ or $\tr(A_{\dot G}^2) - 4\lambda^2+2$ depending on whether or not $x$ and $y$ are adjacent in $G$.
It follows that either $\mu = 0$ or $\mu = 1$.
Furthermore, each vertex of $\dot H$ has degree $\lambda^2$, $\lambda^2-1$, or $\lambda^2-2$.
Therefore, the existence of a signed graph with $\{[-\lambda]^m,[\lambda]^m\}$ implies the existence of a signed graph with one of the two possible spectra: $\{[-\lambda]^{m-2},[0]^2,[\lambda]^{m-2}\}$ or $\{[-\lambda]^{m-2},[-1]^1,[1]^1,[\lambda]^{m-2}\}$ whose vertices have degrees $\lambda^2$, $\lambda^2-1$, or $\lambda^2-2$.
Next we investigate to what extent the converse holds.

We use the symbol $\oplus$ to denote the direct sum of matrices.
We will need the following technical lemma.

\begin{lemma}
    \label{lem:eigenspace}
    Suppose $n \geqslant 3$ and $M$ is a $\mathbb Z$-matrix with spectrum $\{[\lambda]^2,[0]^{n-2}\}$ and whose diagonal entries belong to the set $\{0,1,2\}$.
    Then $M$ is switching isomorphic to one of the following matrices.
    \begin{enumerate}
        \item[(a)] $O_{n-2\lambda} \oplus J_\lambda \oplus J_\lambda$;
        \item[(b)] $O_{n-3\lambda/2} \oplus J_\lambda \oplus 2J_{\lambda/2}$;
        \item[(c)] $O_{n-\lambda} \oplus 2J_{\lambda/2} \oplus 2J_{\lambda/2}$;
        \item[(d)]
        \[
O_{n-\lambda} \oplus \begin{bmatrix}
    2 & 1 &  1 \\
    1 & 2 & -1 \\
    1 & -1 & 2
\end{bmatrix} \otimes J_{\lambda/3};
\]
        \item[(e)]     \[
O_{d_0} \oplus \begin{bmatrix}
    2J_{d_2/2} & O & J & J \\
    O & 2J_{d_2/2} & J & -J \\
    J & J & J_{d_1/2} & O \\
    J & -J & O & J_{d_1/2}
\end{bmatrix},
\]
    \end{enumerate}
where $d_0$, $d_1$, and $d_2$ are the numbers of diagonal entries of $M$ equal to $0$, $1$, and $2$ respectively.
\end{lemma}
\begin{proof}
    First note that $M^2 = \lambda M$ and that $\operatorname{tr}(M) = 2\lambda = d_1 + 2d_2$.
    Partition $\{1,2,\dots,n\}$ in to parts $V_0$, $V_1$, $V_2$, where each $V_i = \{ j \; |\; M[j,j] = i\}$.
    
    Suppose that there exist $x \in V_2$ and $y \in V_2$ such that $M[x,y] = \pm 1$.
    Without loss of generality, we may assume that $M[x,y]=1$.
    Up to switching isomorphism, there are just two $3\times 3$ positive semidefinite $\mathbb Z$-matrices with rank at most $2$ whose leading $2 \times 2$ principal submatrix is $\left [\begin{smallmatrix}2 & 1 \\ 1& 2 \end{smallmatrix}\right ]$ and whose diagonal entries belong to $\{1,2\}$:
    \[
    \begin{bmatrix}
        2 & 1 & 1 \\
        1 & 2 & -1 \\
        1 & -1 & 2
    \end{bmatrix},
            \quad
        \text{ and }
        \quad
    \begin{bmatrix}
        2 & 1 & 2 \\
        1 & 2 & 1 \\
        2 & 1 & 2
    \end{bmatrix}.
    \]
    Therefore, for all $u \in V_2$ and $v \in V_2$, we must have $M[u,v] = \pm 1$ or $M[u,v] = \pm 2$.
    Furthermore, $d_1$ must be equal to $0$.
    Thus $d_2= \lambda$.
    Let $k$ be the number of elements $v \in V_2$ such that $M[u,v] = \pm 2$.
    Then $2\lambda = 4k + (\lambda-k)$, which implies that $k = \lambda/3$.
    Since $M$ has a positive eigenvalue with multiplicity $2$, there must exist $z \in V_2$ such that $M[x,z] = \pm 1$ and $M[y,z]=\pm 1$.
    It follows that $M$ has the form $O_{n-\lambda} \oplus N$, where 
        \[
 N = \begin{bmatrix}
    2J_{a} & J & J  \\
    J & 2J_{b} & -J  \\
    J & -J & 2J_{c} 
\end{bmatrix},
\]
where $a + b + c = \lambda$.
   The only way we can satisfy $M^2 = \lambda M$ is when $a = b = c = \lambda/3$.
   Thus, we must have
            \[
M = O_{n-\lambda} \oplus \begin{bmatrix}
    2 & 1 &  1 \\
    1 & 2 & -1 \\
    1 & -1 & 2
\end{bmatrix} \otimes J_{\lambda/3}.
\]

    Next suppose that there exist $w \in V_1$ and $y \in V_2$ such that $M[w,y] \ne 0$.
    Without loss of generality, we may assume that $M[w,y]=1$.
    Up to switching isomorphism, there are just four $3\times 3$ positive semidefinite $\mathbb Z$-matrices with rank at most $2$ whose leading $2 \times 2$ principal submatrix is $\left [\begin{smallmatrix}1 & 1 \\ 1& 2 \end{smallmatrix}\right ]$ and whose diagonal entries belong to $\{1,2\}$:
    \[
    \begin{bmatrix}
        1 & 1 & 1 \\
        1 & 2 & 1 \\
        1 & 1 & 1
    \end{bmatrix},
    \quad
    \begin{bmatrix}
        1 & 1 & 0 \\
        1 & 2 & 1 \\
        0 & 1 & 1
    \end{bmatrix},
        \quad
    \begin{bmatrix}
        1 & 1 & 1 \\
        1 & 2 & 2 \\
        1 & 2 & 2
    \end{bmatrix},
        \quad
        \text{ and }
        \quad
    \begin{bmatrix}
        1 & 1 & 1 \\
        1 & 2 & 0 \\
        1 & 0 & 2
    \end{bmatrix}.
    \]
    Therefore, for all $u \in V_1$ and $v \in V_2$, we must have $M[u,v] = \pm 1$.
    
    Next we claim that there must exist $x \in V_1$ such that $M[w,x] = 0$.
    Suppose for a contradiction that, for all $u \in V_1$ we have $M[u,w] = \pm 1$.
    Then the {$(w,w)$}-entry of $M^2$ would equal $d_1 + d_2$, which is strictly larger than $\lambda = d_1/2+d_2$ since both $d_1 \geqslant 1$ and $d_2 \geqslant 1$.
    Thus we have established a contradiction.
    A similar argument shows that there must also exist $z \in V_2$ such that $M[y,z] = 0$.
    
    This means that, up to switching isomorphism, $M$ must contain the submatrix
    \[
    S= \begin{bmatrix}
    2 & 0 & 1 & 1 \\
    0 & 2 & 1 & -1 \\
    1 & 1 & 1 & 0 \\
    1 & -1 & 0 & 1
\end{bmatrix},
\]
    which is, up to switching isomorphism, the only $4 \times 4$ positive semidefinite $\mathbb Z$-matrix with rank at most $2$, whose leading $3 \times 3$ principal submatrix is $\left [\begin{smallmatrix}2 & 0 & 1 \\ 0& 2 & 1 \\ 1 & 1 & 1 \end{smallmatrix}\right ]$ and has $(4,4)$-entry equal to $1$.
    
    Up to switching isomorphism, there are just four $5\times 5$ positive semidefinite $\mathbb Z$-matrices with rank at most $2$ whose leading $4 \times 4$ principal submatrix is $S$ and whose diagonal entries belong to $\{1,2\}$:
    \begin{align*}
    &\begin{bmatrix}
    2 & 0 & 1 & 1 & 1 \\
    0 & 2 & 1 & -1 & -1\\
    1 & 1 & 1 & 0 & 0 \\
    1 & -1 & 0 & 1 & 1 \\
    1 & -1 & 0 & 1 & 1
\end{bmatrix},
    \quad
    \begin{bmatrix}
    2 & 0 & 1 & 1 & 1\\
    0 & 2 & 1 & -1 & 1\\
    1 & 1 & 1 & 0 & 1\\
    1 & -1 & 0 & 1 & 0\\
    1 & 1 & 1 & 0 & 1
\end{bmatrix}, \\
    &\begin{bmatrix}
    2 & 0 & 1 & 1 & 0\\
    0 & 2 & 1 & -1 & 2\\
    1 & 1 & 1 & 0 & 1\\
    1 & -1 & 0 & 1 & -1 \\
    0 & 2 & 1 & -1 & 2
\end{bmatrix},
        \quad
    \begin{bmatrix}
    2 & 0 & 1 & 1 & 2\\
    0 & 2 & 1 & -1 & 0\\
    1 & 1 & 1 & 0 & 1\\
    1 & -1 & 0 & 1 & 1 \\
    2 & 0 & 1 & 1 & 2
\end{bmatrix}.
    \end{align*}
    
    It follows that $M$ has the form $O_{d_0} \oplus N$, where 
        \[
 N = \begin{bmatrix}
    2J_{a} & O & J & J \\
    O & 2J_{b} & J & -J \\
    J & J & J_{c} & O \\
    J & -J & O & J_{d}
\end{bmatrix},
\]
and $a+b = d_2$ and $c+d = d_1$.
The nonzero eigenvalues of $N$ are $\lambda \pm \sqrt{4(a-b)^2+(c-d)^2}$.
Therefore, we must have $a = b = d_2/2$ and $c = d = d_1/2$.

Otherwise, if $M[w,y] = 0$ for all $w \in V_1$ and $y \in V_2$ then it is straightforward to deduce that $M$ is switching isomorphic to  $O_{d_0} \oplus J_\lambda \oplus J_\lambda$, $O_{d_0} \oplus J_\lambda \oplus 2J_{\lambda/2}$, or $O_{d_0} \oplus 2J_{\lambda/2} \oplus 2J_{\lambda/2}$.
\end{proof}

Now we are ready to show a sufficient condition on the degrees of a signed graph with spectrum $\{[-\lambda]^{m-2},[-1]^1,[1]^1,[\lambda]^{m-2}\}$ that ensures it is an induced subgraph of a signed graph with spectrum $\{[-\lambda]^{m-1},[0]^1,[\lambda]^{m-1}\}$.

\begin{theorem}
    \label{thm:extend41}
    Let $\dot G$ be a signed graph with spectrum $\{[-\lambda]^{m-2},[-1]^1,[1]^1,[\lambda]^{m-2}\}$.
    Suppose the vertices of $G$ have degrees $\lambda^2$, $\lambda^2-1$, or $\lambda^2-2$, with at least one having degree $\lambda^2-1$.
    Then $\dot G$ can be extended by a vertex to obtain a signed graph with spectrum $\{[-\lambda]^{m-1},[0]^1,[\lambda]^{m-1}\}$ and vertex-degrees $\lambda^2$ or $\lambda^2-1$.
\end{theorem}
\begin{proof}
    Let $A$ be the adjacency matrix of $\dot G$.
    The matrix $M := \lambda^2I-A^2$ is a $\mathbb Z$-matrix with spectrum $\{[\lambda^2-1]^{2},[0]^{2m-4}\}$ and all diagonal entries equal to $0$, $1$, or $2$. 
    By Lemma~\ref{lem:eigenspace}, $M$ is switching isomorphic to $O_{d_0} \oplus J_{\lambda^2-1} \oplus J_{\lambda^2-1}$,
     $O_{d_0} \oplus J_{\lambda^2-1} \oplus 2J_{(\lambda^2-1)/2}$, or
    \[
O_{d_0} \oplus \begin{bmatrix}
    2J_{d_2/2} & O & J & J \\
    O & 2J_{d_2/2} & J & -J \\
    J & J & J_{d_1/2} & O \\
    J & -J & O & J_{d_1/2}
\end{bmatrix},
\]
where $d_i$ denotes the number of vertices of $G$ of degree $\lambda^2-i$.
   In each case, we observe that there are two mutually orthogonal $\{0,\pm 1\}$-vectors $\mathbf x$ and $\mathbf y$ in the $(\lambda^2-1)$-eigenspace $\mathcal E$ of $M$ and we can assume that $\mathbf x^\transpose \mathbf x = \lambda^2-1$.
    
        Since $A\mathbf x$ is in $\mathcal E$, we can write $A\mathbf x = a\mathbf x + b\mathbf y$.
    Then
    $(A \mathbf x)^\transpose A \mathbf x = (\lambda^2-1) = a^2(\lambda^2-1)+b^2\mathbf y^\transpose \mathbf y$.
    Since $A\mathbf x$ has integer entries and the supports of $\mathbf x$ and $\mathbf y$ are distinct for each of the cases above, we must have $a, b \in \mathbb Z$.
    Thus, either $(a,b) = (\pm 1,0)$ or $(0,\pm 1)$.
    In either case, note that $A\mathbf x$ is a $\{0,\pm 1 \}$-vector.
    
    Next we show that $\mathbf x$ cannot be a $\pm 1$-eigenvector of $A$.
    Observe that, in each of the three cases for $M$ above, the row vector $\pm \mathbf x^\transpose$ is a row (say the {$i$}th row) of the matrix $M$.
    The diagonal entries of $A^3$ are all even and $AM = \lambda^2A - A^3$.
    If $\mathbf x$ were an $\pm 1$-eigenvector of $A$ then the {$(i,i)$}-entry of $AM$ would be odd, which implies the same for $A^3$.
    Therefore $A\mathbf x = \pm \mathbf y$.
    Furthermore, $A\mathbf x$ is orthogonal to $\mathbf x$.
    This rules out the possibility that $M$ were switching isomorphic to $O_{d_0} \oplus J_{\lambda^2-1} \oplus 2J_{(\lambda^2-1)/2}$.
    Indeed, in this case, there is no $\{0,\pm 1 \}$-vector $\mathbf z$ in $\mathcal E$ with $\mathbf z^\transpose \mathbf z$ that is orthogonal to $\mathbf x$.
    Now form the matrix
    \[
    B = \begin{bmatrix} 0 & \mathbf x^\transpose \\
                        \mathbf x & A
    \end{bmatrix},
    \]
    which is the adjacency matrix of a signed graph $\dot H$.
    The vector $( 0, \mathbf v^\transpose)^\transpose$ is a $\pm \lambda$-eigenvector for $B$ when $\mathbf v$ is a $\pm \lambda$-eigenvector of $A$.
    
    Furthermore, the vector $( -1, \mathbf x^\transpose A)^\transpose$ is a $0$-eigenvector for $B$.
    Since the trace of $B$ is $0$, we see that the spectrum of $B$ is equal to $\{[-\lambda]^{m-2},[-\alpha]^1,[0]^1,[\alpha]^1,[\lambda]^{m-2}\}$, for some $\alpha \in \mathbb R$.
    On the other hand, the trace of $B^2$ is equal to $\mathbf x^\transpose \mathbf x + \operatorname{tr}(\mathbf x \mathbf x^\transpose + A^2)$, which equals $2(\lambda^2-1) + (2m-4)\lambda^2+2 = (2m-2)\lambda^2$.
    Thus, $\alpha = \lambda$.
    Lastly, it is straightforward to check that the degrees of the vertices of $\dot H$ are $\lambda^2$ and $\lambda^2-1$.
\end{proof}

In order to strengthen Proposition~\ref{thm:extend41} further, by examining its proof, we are led to the following questions.

\begin{question}
Let $\dot G$ be a signed graph with spectrum $\{[-\lambda]^{m-2},[-1]^1,[1]^1,[\lambda]^{m-2}\}$.
Suppose the vertices of $G$ have degrees $\lambda^2$, $\lambda^2-1$, or $\lambda^2-2$.
Is it possible for $\lambda^2 I - A_{\dot G}^2$ to be switching isomorphic to
\[O_{2m-1-\lambda^2} \oplus \begin{bmatrix}
    2 & 1 &  1 \\
    1 & 2 & -1 \\
    1 & -1 & 2
\end{bmatrix} \otimes J_{(\lambda^2-1)/3}?
\]
\end{question}

\begin{question}
Let $\dot G$ be a signed graph with spectrum $\{[-\lambda]^{m-2},[-1]^1,[1]^1,[\lambda]^{m-2}\}$.
Suppose the vertices of $G$ have degrees $\lambda^2$, $\lambda^2-1$, or $\lambda^2-2$.
If $\lambda^2 I - A_{\dot G}^2 = O_{2m-1-\lambda^2} \oplus 2J_{(\lambda^2-1)/2} \oplus 2J_{(\lambda^2-1)/2}$ then does it imply that $\dot G$ is an induced subgraph of a signed graph with spectrum $\{[-\lambda]^{m},[\lambda]^{m}\}$?
\end{question}



Next we show a sufficient condition on the degrees of a signed graph with spectrum $\{[-\lambda]^{m-2},[0]^2,[\lambda]^{m-2}\}$ that ensures it is an induced subgraph of a signed graph with spectrum $\{[-\lambda]^{m-1},[0]^1,[\lambda]^{m-1}\}$.

\begin{theorem}
\label{thm:extend40}
    Let $\dot G$ be a signed graph with spectrum $\{[-\lambda]^{m-2},[0]^2,[\lambda]^{m-2}\}$.
     Suppose the vertices of $G$ have degrees $\lambda^2$, $\lambda^2-1$, or $\lambda^2-2$, with at least $\lambda^2+1$ having degree $\lambda^2-1$.
    Then $\dot G$ can be extended by a vertex to obtain a signed graph with spectrum $\{[-\lambda]^{m-1},[0]^1,[\lambda]^{m-1}\}$ and vertex-degrees $\lambda^2$ or $\lambda^2-1$.
\end{theorem}
\begin{proof}
    The matrix $M := \lambda^2I-A_{\dot G}^2$ is a $\mathbb Z$-matrix with spectrum $\{[\lambda^2]^{2},[0]^{2m-4}\}$ and all diagonal entries equal to $0$, $1$, or $2$. 
    Partition $\{1,2,\dots,n\}$ in to parts $V_0$, $V_1$, $V_2$, where each $V_i = \{ j \; |\; M[j,j] = i\}$.
     By Lemma~\ref{lem:eigenspace}, $M$ is switching isomorphic to $O_{d_0} \oplus J_{\lambda^2} \oplus J_{\lambda^2}$ or
    \[
O_{d_0} \oplus \begin{bmatrix}
    2J_{d_2/2} & O & J & J \\
    O & 2J_{d_2/2} & J & -J \\
    J & J & J_{d_1/2} & O \\
    J & -J & O & J_{d_1/2}
\end{bmatrix},
\]
    where $d_i$ denotes the number of vertices of $G$ of degree $\lambda^2-i$.
    In each case, we observe that there are two mutually orthogonal $\{0,\pm 1\}$-vectors $\mathbf x$ and $\mathbf y$ in the $\lambda^2$-eigenspace $\mathcal E$ of $M$.
    Furthermore, we can assume that $\mathbf x^\transpose \mathbf x = \lambda^2$ and that the set $V_2$ is contained in the support of $\mathbf x$.
    Now form the matrix
    \[
    B = \begin{bmatrix} 0 & \mathbf x^\transpose \\
                        \mathbf x & A_{\dot G}
    \end{bmatrix},
    \]
    which is the adjacency matrix of a signed graph $\dot H$.
    Note that $( 0, \mathbf v^\transpose)^\transpose$ is a $\pm \lambda$-eigenvector for $B$ when $\mathbf v$ is a $\pm \lambda$-eigenvector of $A$.
    Furthermore, $(0, \mathbf w^\transpose)^\transpose$ is in the nullspace of $B$ where $\mathbf w$ is a vector in the nullspace of $A$ that is orthogonal to $\mathbf x$.
    Since the trace of $B$ is $0$, we see that $B$ has spectrum $\{[-\lambda]^{m-2},[-\mu]^1,[0]^1,[\mu]^1,[\lambda]^{m-2}\}$, for some $\mu \in \mathbb R$.
    Since the trace of $B^2$ is $(2m-2)\lambda^2$, it follows that $\mu = \lambda$.
    It remains to check the degrees of $\dot H$ are $\lambda^2$ and $\lambda^2-1$. This follows since the support of $\mathbf x$ contains the set $V_2$.
\end{proof}

The next example shows that there exist signed graphs that possess the spectrum $\{[-\lambda]^{m-2},[0]^2,[\lambda]^{m-2}\}$ such that every vertex has degree $\lambda^2$, $\lambda^2-1$, or $\lambda^2-2$ but that are not induced subgraphs of a signed graph with spectrum $\{[-\lambda]^{m},[\lambda]^{m}\}$.

\begin{example}
\label{ex:ce}
Let $\dot G$ be a signed graph with underlying graph $G = K_{1,3}$.
Then $\dot G$ has spectrum $\{ [-\sqrt{3}]^1,[0]^2, [\sqrt{3}]^1\}$.
However, $\dot G$ is not an induced subgraph of a signed graph with spectrum $\{ [-\sqrt{3}]^3, [\sqrt{3}]^3 \}$.
Note that $3I-A_{\dot G}^2$ is switching isomorphic to 
\[  O_{1} \oplus \begin{bmatrix}
    2 & 1 &  1 \\
    1 & 2 & -1 \\
    1 & -1 & 2
\end{bmatrix}.
\]
\end{example}

Example~\ref{ex:ce} is the only example we know of a signed graph $\dot G$ having spectrum $\{[-\lambda]^{m-2},[0]^2,[\lambda]^{m-2}\}$ such that $\lambda^2 I - A_{\dot G}^2$ is switching isomorphic to
        \[
O_{2m-2-\lambda^2} \oplus \begin{bmatrix}
    2 & 1 &  1 \\
    1 & 2 & -1 \\
    1 & -1 & 2
\end{bmatrix} \otimes J_{\lambda^2/3}.
\]
Thus we ask the following question.

\begin{question}
Let $\dot G$ be a signed graph with spectrum $\{[-\lambda]^{m-2},[0]^2,[\lambda]^{m-2}\}$.
Suppose the vertices of $G$ have degrees $\lambda^2$, $\lambda^2-1$, or $\lambda^2-2$.
Suppose that $G$ is not $K_{1,3}$.
Is it possible for $\lambda^2 I - A_{\dot G}^2$ to be switching isomorphic to
\[O_{2m-2-\lambda^2} \oplus \begin{bmatrix}
    2 & 1 &  1 \\
    1 & 2 & -1 \\
    1 & -1 & 2
\end{bmatrix} \otimes J_{\lambda^2/3}?
\]
\end{question}

The assumptions in Theorem~\ref{thm:extend40} are stronger than those of Theorem~\ref{thm:extend41}.
This is in order to avoid the case when $\lambda^2 I - A^2$ is switching isomorphic to $ O_{2m-2-\lambda^2} \oplus J_{\lambda^2} \oplus 2J_{\lambda^2/2}$, for which the resulting signed graph $\dot H$ does not have the required vertex degrees.
We do not know if such pathological signed graphs exist, hence we ask the following question.

\begin{question}
Let $\dot G$ be a signed graph with spectrum $\{[-\lambda]^{m-2},[0]^2,[\lambda]^{m-2}\}$.
Suppose the vertices of $G$ have degrees $\lambda^2$, $\lambda^2-1$, or $\lambda^2-2$.
Is it possible for $\lambda^2 I - A_{\dot G}^2$ to be switching isomorphic to $ O_{2m-2-\lambda^2} \oplus J_{\lambda^2} \oplus 2J_{\lambda^2/2}$?
\end{question}




\section{Signed $(0,2)$-graphs with a symmetric spectrum of three or four eigenvalues}
\label{sec:few}

In contrast to those with a symmetric spectrum of two eigenvalues, in this section we show that the variety of $(0,2)$-graphs with a symmetric spectrum of three or four eigenvalues is severely limited.
The main result of this section is that every signed $(0,2)$-graph with a symmetric spectrum of three eigenvalues has the complete bipartite graph $K_{2,2}$ as its underlying graph.
We also classify the signed  $(0,2)$-graphs having a particular symmetric spectrum of four eigenvalues.

    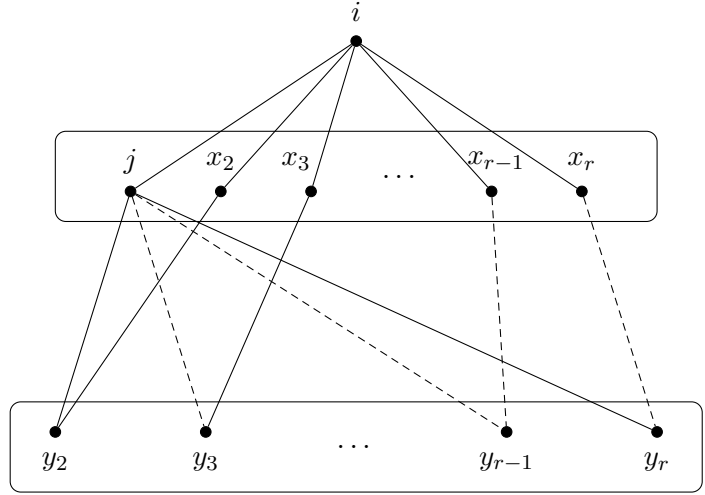
\begin{figure}[htbp]
        \centering
        \begin{center}
    \begin{tikzpicture}[scale=2, auto]
\begin{scope}
    \node[vertex] (i) at (-1,3) {};
    \node at (-1,3.2) {$i$};
	\foreach \pos/\name in {{(-2.5,2)/y1}, {(-1.9,2)/y2}, {(-1.3,2)/y3}, {(-0.1,2)/yk1}, {(0.5,2)/ys}}
		\node[vertex] (\name) at \pos {};
	\foreach \pos/\name in {{(-2.5,2.2)/j}, {(-1.9,2.2)/x_2}, {(-1.4,2.2)/x_3}, {(-0.07,2.2)/x_{r-1}}, {(0.5,2.2)/x_{r}}}
		\node at \pos {$\name$};
	\node at (-0.7,2.1) {$\dots$};
	\draw[rounded corners] (-3, 1.8) rectangle (1, 2.4) {};
	\foreach \pos/\name in {{(-3,0.4)/yn1}, {(-2,0.4)/yn2}, {(0,0.4)/yin1}, {(1,0.4)/yin2}}
		\node[vertex] (\name) at \pos {};
	\foreach \pos/\name in {{(-3,0.2)/y_2}, {(-2,0.2)/y_3}, {(0,0.2)/y_{r-1}}, {(1,0.2)/y_{r}}}
		\node at \pos {$\name$};
	\node at (-1,0.3) {$\dots$};
	\draw[rounded corners] (-3.3, -0) rectangle (1.3, 0.6) {};
	\foreach \edgetype/\source/ \dest in {pedge/i/y1, pedge/i/y2, pedge/i/y3, pedge/i/yk1, pedge/i/ys, pedge/y1/yn1, nedge/y1/yn2, nedge/y1/yin1, pedge/y1/yin2, pedge/y2/yn1, pedge/y3/yn2, nedge/yk1/yin1, nedge/ys/yin2}
	\path[\edgetype] (\source) -- (\dest);
\end{scope}
\end{tikzpicture}
\end{center}
        \caption{The vertices $i$, $j$, and their neighbours.}
        \label{fig:diag1}
    \end{figure}

\begin{theorem}
    Let $\dot G = (G, \sigma)$ be a signed $(0,2)$-graph with spectrum $\{[-\lambda]^m,[0]^d,[\lambda]^m\}$.
    Then $G$ is the complete bipartite graph $K_{2,2}$.
\end{theorem}
\begin{proof}
    Let $A$ be the adjacency matrix of $\dot G$.
    Then $A^3 = \lambda^2A$.
    By Lemma~\ref{lem:reg}, the underlying graph $G$ of $\dot G$ is regular with degree $r$ (say).
    Consider vertices $i \sim j$, let $x_1 = j$, $x_2$, $\dots$, $x_r$ denote the neighbours of $i$ and  $y_1 = i$, $y_2$, $\dots$, $y_r$ denote the neighbours of $j$.
    By Theorem~\ref{thm:trianglefree}, $\dot G$ is triangle free.
    Therefore, the sets $\{x_1,x_2,\dots,x_r\}$ and $\{y_1,y_2,\dots,y_r\}$ have no intersection.
    Without loss of generality, we can assume that $\sigma(i,x_l) = 1$ for all $l \in \{1,2,\dots,r\}$ (see Figure~\ref{fig:diag1}).
    Set 
    $$q_{\pm} := |\{ l \; | \; l \in \{2, 3, \dots r\} \text{ and } \sigma(j,y_l)\sigma(y_l,x_l) = \pm 1\}|.$$
    Then $q_{+} + q_{-} = {r}-1$.
    Observe that the $(i,j)$-entry of $A^3$ can be expressed as $2r-1+q_+ -q_- = r+2q_+$.
    On the other hand, the $(i,j)$-entry of $A^3$ equals $\lambda^2$.
    Thus $\lambda^2 = r + 2q_+$.
    
    Next consider the entries of $A^2$ along the $i$th row.
    Observe that, for $l \ne i$, the $(i,l)$-entry of $A^2$ is equal to $\pm 2$ only if $l = y_m$ for some $m \in \{2,3,\dots,r\}$.
    Furthermore, each diagonal entry of $A^2$ is equal to $r$.
    Since $A^4 = \lambda^2A^2$, we have that $\lambda^2 r = r^2 + 4q_+$.
    Combining this with $\lambda^2 = r + 2q_+$ yields $(\lambda^2-r)(r-2) = 0$.
    
    Since $d > 0$, we must have $r \ne \lambda^2$, otherwise $A^2 = rI$, a contradiction.
    Thus we must have $r = 2$.
    This implies the underlying graph of $\dot G$ must be a cycle, since $G$ is connected by definition.
    And the $4$-cycle $K_{2,2}$ is the only cycle that is a $(0,2)$-graph.
\end{proof}

Next we consider signed $(0,2)$-graphs having spectrum $\{[-\lambda]^m,[-\mu]^1,[\mu]^1,[\lambda]^m\}$, where $\mu \geqslant 1$.
We will require the following lemma.

\begin{lemma}
    \label{lem:struc2}
   Suppose $n \geqslant 3$ and $M$ is  an $n\times n$  $\mathbb Z$-matrix with a constant diagonal, off-diagonal entries in $\{0,\pm 2\}$ and spectrum $\{[\lambda]^2,[0]^{n-2}\}$ for some $\lambda > 0$.
    Then $M$ is switching isomorphic to $2J_{n/2} \oplus 2J_{n/2}$ and $\lambda = n$.
\end{lemma}
\begin{proof}
    Suppose each diagonal entry of $M$ is $d$.
    First note that $d > 0$, since $\lambda > 0$.
    By interlacing, every $3\times 3$ principal submatrix must be singular.
    Up to switching isomorphism, there are five possibilities for such submatrices:
    \begin{align*}
            \begin{bmatrix}
        d & 0 & 0 \\
        0 & d & 0 \\
        0 & 0 & d
    \end{bmatrix},
    \quad
    \begin{bmatrix}
        d & 2 & 0 \\
        2 & d & 0 \\
        0 & 0 & d
    \end{bmatrix},
        \quad
    \begin{bmatrix}
        d & 2 & 2 \\
        2 & d & 0 \\
        2 & 0 & d
    \end{bmatrix},
    \quad
    \begin{bmatrix}
        d & 2 & 2 \\
        2 & d & 2 \\
        2 & 2 & d
    \end{bmatrix},
        \quad
    \begin{bmatrix}
        d & 2 & 2 \\
        2 & d & -2 \\
        2 & -2 & d
    \end{bmatrix}.
    \end{align*}
    Out of these, the ones that are singular are 
        \begin{align*}
    \begin{bmatrix}
        d & 2 & 0 \\
        2 & d & 0 \\
        0 & 0 & d
    \end{bmatrix},
        \quad
    \begin{bmatrix}
        d & 2 & 2 \\
        2 & d & 2 \\
        2 & 2 & d
    \end{bmatrix},
        \quad
    \begin{bmatrix}
        d & 2 & 2 \\
        2 & d & -2 \\
        2 & -2 & d
    \end{bmatrix}.
    \end{align*}
    Furthermore, they are only singular when $d =2$.
    Thus, $d$ must be equal to $2$ and the lemma follows from Lemma~\ref{lem:eigenspace}.
\end{proof}



Lastly, we show that the signed tetrahedron $\mathcal T$ (Figure~\ref{fig:T}) is the only signed $(0,2)$-graph with spectrum $\{[-\lambda]^m,[-\mu]^1,[\mu]^1,[\lambda]^m\}$, where {$1 \leqslant \mu < \lambda$}.

\begin{theorem}
    Let $\dot G$ be a signed $(0,2)$-graph with spectrum $\{[-\lambda]^m,[-\mu]^1,[\mu]^1,[\lambda]^m\}$,
    where {$1 \leqslant \mu < \lambda$}.
    Then the underlying graph $G$ is the complete graph $K_4$.
\end{theorem}
\begin{proof}
    Let $A$ be the adjacency matrix of $\dot G$.
    By Lemma~\ref{lem:reg}, the underlying graph $G$ of $\dot G$ is regular with degree $r$.
    The matrix $M := \lambda^2I-A^2$ is a $\mathbb Z$-matrix with all diagonal entries equal to $\lambda^2-r$, off-diagonal entries in $\{0,\pm 2\}$, and spectrum $\{[\lambda^2-\mu^2]^2,[0]^{n-2}\}$. 
    By Lemma~\ref{lem:struc2}, $r = \lambda^2-2$ and $n = \lambda^2-\mu^2$.
    Since $r \leqslant n-1$ and $\mu^2 \geqslant 1$, we must have $r=n-1$.
    Thus $G$ is complete, and the only complete $(0,2)$-graph is $K_4$.
\end{proof}


\section*{Acknowledgements}

The first author was supported in part by the Singapore Ministry of Education Academic Research Fund (Tier 1); grant number RG21/20.
The second author is partially supported by the Serbian Ministry of Education, Science and Technological Development via the University of Belgrade.


\end{document}